\documentclass[tocmacros]{daj}

\dajAUTHORdetails{%
  title = {A Note on the Power Sums of the Number of Fibonacci Partitions}, 
  author = {Carlo Sanna},
  plaintextauthor = {Carlo Sanna},
  plaintexttitle = {A Note on the Power Sums of the Number of Fibonacci Partitions}, 
  keywords = {automaton, asymptotic formula, Fibonacci, power sum, partition},
}   

\dajEDITORdetails{%
   year={2025},
   number={2},
   received={1 October 2023},   
   published={6 May 2025},  
   doi={10.19086/da.137601},       
}   

\usepackage{booktabs}
\usepackage{enumitem}
\DeclareMathOperator{\trace}{tr}
\allowdisplaybreaks

\begin{document}

\begin{frontmatter}[classification=text]

\title{A Note on the Power Sums of the Number of Fibonacci Partitions} 

\author[casa]{Carlo Sanna\thanks{C.~Sanna is a member of the INdAM group GNSAGA and of CrypTO, the group of Cryptography and Number Theory of the Politecnico di Torino}}

\begin{abstract}
    For every nonnegative integer $n$, let $r_F(n)$ be the number of ways to write $n$ as a sum of distinct Fibonacci numbers, where the order of the summands does not matter.
	Moreover, for all positive integers $p$ and $N$, let
	\begin{equation*}
		S_{F}^{(p)}(N) := \sum_{n = 0}^{N - 1} \big(r_F(n)\big)^p .
	\end{equation*}
	Chow, Jones, and Slattery determined the order of growth of $S_{F}^{(p)}(N)$ for $p \in \{1,2\}$.
	
	We prove that, for all positive integers $p$, there exists a real number $\lambda_p > 1$ such that
	\begin{equation*}
		S^{(p)}_F(N) \asymp_p N^{(\log \lambda_p) /\!\log \varphi}
	\end{equation*}
	as $N \to +\infty$, where $\varphi := (1 + \sqrt{5})/2$ is the golden ratio.
	Furthermore, we show that
	\begin{equation*}
		\lim_{p \to +\infty} \lambda_p^{1/p} = \varphi^{1/2} .
	\end{equation*}
	Our proofs employ automata theory and a result on the generalized spectral radius due to Blondel and Nesterov.
\end{abstract}
\end{frontmatter}

\section{Introduction}

For every nonnegative integer $n$, let $r_F(n)$ be the number of ways to write $n$ as a sum of distinct Fibonacci numbers, where the order of the summands does not matter.
Formulas for $r_F(n)$ were proved by Carlitz~\cite{MR0236094}, Robbins~\cite{MR1394758}, Berstel~\cite{MR1922290}, Ardila~\cite{MR2093874}, Weinstein~\cite{MR3499711}, and Chow--Slattery~\cite{MR4235264}.

For all positive integers $p$ and $N$, define the power sum
\begin{equation*}
	S_{F}^{(p)}(N) := \sum_{n = 0}^{N - 1} \big(r_F(n)\big)^p .
\end{equation*}
Chow and Slattery~\cite{MR4235264} proved a recursive formula for $S_{F}^{(1)}(N)$, and obtained the asymptotic estimate
\begin{equation}\label{equ:SF1asymp}
	S_F^{(1)}(N) \asymp N^{(\log 2) /\! \log \varphi}
\end{equation}
as $N \to +\infty$, where $\varphi := (1 + \sqrt{5}) / 2$ is the golden ratio. Moreover, they showed that the limit
\begin{equation*}
	\lim_{N \to +\infty} \frac{S_F^{(1)}(N)}{N^{(\log 2) /\! \log \varphi}}
\end{equation*}
does not exist.
Limits of this kind were then studied in a far more general context by Zhou~\cite{preprintZhou23}.
Chow and Jones~\cite{preprintCJ23} proved an inhomogeneous linear recurrence for a subsequence of $S_F^{(2)}(N)$ and showed that
\begin{equation}\label{equ:SF2asymp}
	S_F^{(2)}(N) \asymp N^{(\log \lambda_2) /\! \log \varphi}
\end{equation}
as $N \to +\infty$, where $\lambda_2 = 2.48\dots$ is the greatest root of the polynomial $X^3 - 2X^2 - 2X + 2$.

Our first result is the following.

\begin{theorem}\label{thm:main}
	Let $p$ be a positive integer.
	Then there exists a real number $\lambda_p > 1$ such that
	\begin{equation}\label{equ:main-asymp}
		S^{(p)}_F(N) \asymp_p N^{(\log \lambda_p) /\!\log \varphi}
	\end{equation}
	as $N \to +\infty$.
\end{theorem}

While the proofs of~\eqref{equ:SF1asymp} and~\eqref{equ:SF2asymp} given by Chow, Jones, and Slattery rely on counting arguments and inequalities, the proof of Theorem~\ref{thm:main} is based on automata theory.
We remark that Berstel~\cite{MR1922290} and Shallit~\cite{MR4195833} had already employed automata theory to study $r_F(n)$.

The proof of Theorem~\ref{thm:main} provides an effective algorithm to compute $\lambda_p$.
More precisely, we have that $\lambda_p$ is the greatest real root of an effectively computable monic polynomial with integer coefficients.
In particular, it follows that $\lambda_p$ is an algebraic integer.
Table~\ref{tab:lambda-p} provides the first values of $\lambda_p$ and their minimal polynomials over $\mathbb{Q}$.

\begin{table}[ht]
	\begin{center}
		\begin{tabular}{c@{\hskip 10pt}c@{\hskip 10pt}l@{\hskip 10pt}}
			\toprule
			$p$ & $\lambda_p$ & \text{minimal polynomial of $\lambda_p$ over $\mathbb{Q}$} \\
			\midrule
			$1$ & $2.00000\dots$ & $X - 2$ \\
			$2$ & $2.48119\dots$ & $X^{3} - 2 X^{2} - 2 X + 2$ \\
			$3$ & $3.08613\dots$ & $X^{3} - 2 X^{2} - 4 X + 2$ \\
			$4$ & $3.84606\dots$ & $X^{5} - 2 X^{4} - 7 X^{3} - 2 X + 2$ \\
			$5$ & $4.80052\dots$ & $X^{5} - 2 X^{4} - 11 X^{3} - 8 X^{2} - 20 X + 10$ \\
			$6$ & $5.99942\dots$ & $X^{7} - 2 X^{6} - 17 X^{5} - 28 X^{4} - 88 X^{3} + 26 X^{2} - 4 X + 4$ \\
			$7$ & $7.50569\dots$ & $X^{7} - 2 X^{6} - 26 X^{5} - 74 X^{4} - 311 X^{3} + 34 X^{2} - 84 X + 42$ \\
			$8$ & $9.39867\dots$ & $X^{9} - 2 X^{8} - 40 X^{7} - 174 X^{6} - 969 X^{5} - 2 X^{4} - 428 X^{3} + 174 X^{2} - 4 X + 4$ \\
			\bottomrule
		\end{tabular}
		\vspace{5pt}
	\end{center}
	\caption{First values of $\lambda_p$ and their minimal polynomials over $\mathbb{Q}$.}\label{tab:lambda-p}
\end{table}

Our second result regards the growth of $\lambda_p$.

\begin{theorem}\label{thm:lambda-p}
	We have
	\begin{equation*}
		\lim_{p \to +\infty} \lambda_p^{1/p} = \varphi^{1/2} .
	\end{equation*}
\end{theorem}

We remark that the methods used in the proofs of Theorems~\ref{thm:main} and~\ref{thm:lambda-p} can be adapted to prove similar results, where the sequence of Fibonacci numbers is replaced by other linear recurrences, such as the sequence of Lucas numbers.

\subsection*{Notation}

We employ the standard asymptotic notation.
More precisely, given two functions $f, g \colon \mathbb{N} \to \mathbb{R}$, we write $f(n) \sim g(n)$ and $f(n) \asymp g(n)$ to mean
\begin{equation}\label{equ:liminfsup}
	\lim_{n \to +\infty} \frac{f(n)}{g(n)} = 1 \quad\text{and}\quad
	0 < \liminf_{n \to +\infty} \frac{f(n)}{g(n)} \leq \limsup_{n \to +\infty} \frac{f(n)}{g(n)} < +\infty ,
\end{equation}
respectively.
We add parameters as subscripts of the symbol ``$\asymp$'' if the values of the limit inferior and limit superior in~\eqref{equ:liminfsup} may depend on such parameters.
Other notation are introduced when first used.

\section{Preliminaries}

In this section, we collect some preliminary results needed later.
Sections~\ref{sec:zeck} and~\ref{sec:berstel} follow most of the notation of Shallit~\cite[Section~3]{MR4195833}.

\subsection{Fibonacci representation}\label{sec:zeck}

Let $(f_n)$ be the sequence of Fibonacci numbers, recursively defined by $f_1 := 1$, $f_2 := 2$, and $f_{n + 2} = f_{n + 1} + f_n$ for every positive integer $n$.
Note that, differently from the usual definition, we shifted the indices so that the sequence of Fibonacci numbers is strictly increasing, which is more convenient for our purposes.
For every binary word $x = x_1 \cdots x_\ell$ of $\ell$ bits $x_1, \dots, x_\ell \in \{0,1\}$, let $[x]_F := \sum_{i=1}^\ell x_i f_{\ell - i + 1}$.
Moreover, let $C_F$ be the set of binary words having no consecutive $1$'s and that do not start with $0$.
It is well known that every nonnegative integer can be written in a unique way as a sum of distinct nonconsecutive Fibonacci numbers~\cite{MR0308032}.
Hence, for every integer $n \geq 0$ there exists a unique $y \in C_F$ such that $n = [y]_F$.
The word $y$ is called the \emph{canonical Fibonacci representation} of $n$.

\subsection{Berstel's automaton}\label{sec:berstel}

For binary words $x=x_1 \cdots x_\ell$ and $y = y_1 \cdots y_\ell$ of the same length, define $x \times y$ to be the word of pairs $(x_1, y_1) \cdots (x_\ell, y_\ell)$.
Berstel~\cite{MR1922290} constructed a deterministic finite automaton $\mathcal{B}$ that takes as input $x \times y$ and accepts if and only if
\begin{equation*}
	[x]_F = [y]_F \quad\text{and}\quad y \in 0^* C_F .
\end{equation*}
Berstel's automaton has $4$ states, which we call $a,b,c,d$, and is depicted in Figure~\ref{fig:berstel-automaton}.

\begin{figure}[ht]
	\centering
	\includegraphics{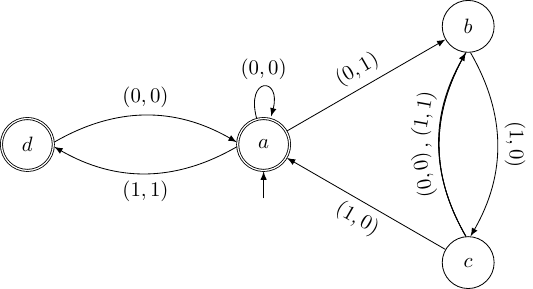}
	\caption{Berstel's automaton $\mathcal{B}$.}
	\label{fig:berstel-automaton}
\end{figure}

\subsection{Irreducible aperiodic automata}

We need to recall some terminology related to the Perron--Frobenius theorem and its application to automata theory~(cf.~\cite[~Section~V.5.2]{MR2483235}).

Let $M$ be an $n \times n$ matrix with nonnegative real entries.
The \emph{dependency graph} of $M$ is the (directed) graph having vertex set $\{1, \dots, n\}$ and edge set $\{i \to j : M_{i,j} \neq 0\}$.
The matrix $M$ is \emph{irreducible} if its dependency graph is strongly connected (that is, any two vertices are connected by a directed path).
A strongly connected graph is \emph{aperiodic} if its vertex set $V$ cannot be partitioned into disjoint sets $V_0, \dots, V_{m - 1}$, with $m > 1$, such that each edge with source in $V_i$ has destination in $V_{i + 1 \bmod m}$.
The matrix $M$ is \emph{aperiodic} if its dependency graph is aperiodic.

As a consequence of the Perron--Frobenius theorem~\cite[V.34]{MR2483235}, if the matrix $M$ is irreducible and aperiodic, then the spectrum of $M$ contains a unique eigenvalue $\lambda$ of maximum absolute value, which is called the \emph{Perron--Frobenius eigenvalue} of $M$.
Moreover, the eigenvalue $\lambda$ is real and simple.
In particular, we have that $\lambda$ is equal to the spectral radius of $M$.

\begin{theorem}\label{thm:perron-frobenius-automaton}
	Let $\mathcal{A}$ be a deterministic finite automaton whose graph is strongly connected and aperiodic.
	Then the number $A_\ell$ of words of length $\ell$ that are accepted by $\mathcal{A}$ satisfies
	\begin{equation*}
		A_\ell \sim c \lambda^\ell
	\end{equation*}
	as $\ell \to +\infty$, where $c > 0$ is a real number and $\lambda$ is the Perron--Frobenius eigenvalue of the transition matrix of $\mathcal{A}$.
\end{theorem}
\begin{proof}
	See~\cite[Proposition~V.7]{MR2483235}.
\end{proof}

\subsection{Generalized spectral radius}

Let $M$ be a square matrix over $\mathbb{C}$.
We let $\rho(M)$ denote the \emph{spectral radius} of $M$, that is, the maximum of the absolute values of the eigenvalues of $M$.
Let $\Sigma$ be a finite set of $n \times n$ matrices over $\mathbb{C}$.
Daubechies and Lagarias~\cite[p.~235]{MR1142737} defined the \emph{generalized spectral radius} of $\Sigma$ as
\begin{equation*}
	\rho(\Sigma) := \limsup_{k \to +\infty} \big(\rho_k(\Sigma)\big)^{1/k} ,
\end{equation*}
where
\begin{equation*}
	\rho_k(\Sigma) := \max\!\left\{ \,\rho\!\left(\prod_{i=1}^k M_i \right) : M_1, \dots, M_k \in \Sigma \right\} .
\end{equation*}
Note that, if $\Sigma$ contains a single matrix $M$, then $\rho(\Sigma) = \rho(M)$.
If $\Sigma = \{M_1, \dots, M_h\}$, we also write $\rho(M_1, \dots, M_h) := \rho(\Sigma)$.

For every matrix $M$ over $\mathbb{C}$ and for every positive integer $k$, let
\begin{equation*}
	M^{\otimes k} := \underbrace{M \otimes \cdots \otimes M}_{\text{$k$ times $M$}} ,
\end{equation*}
where $\otimes$ denotes the Kronecker product of matrices.

\begin{theorem}\label{thm:kronecker-power}
	Let $M_1, \dots, M_h$ be $n \times n$ matrices with real nonnegative entries.
	Then
	\begin{equation}\label{equ:limsup-kronecker}
		\lim_{k \to +\infty} \left(\rho\big(M_1^{\otimes k} + \cdots + M_h^{\otimes k}\big)\right)^{1/k} = \rho(M_1, \dots, M_h) .
	\end{equation}
\end{theorem}
\begin{proof}
	This is a result of Blondel and Nesterov~\cite{MR2176820}.
	Note that they stated~\eqref{equ:limsup-kronecker} for the \emph{joint spectral radius} instead of the generalized spectral radius.
	However, for finite sets of matrices the joint spectral radius and the generalized spectral radius are equal (see~\cite[Theorem~IV]{MR1152485} or \cite[Theorem~1$^\prime$]{MR1334574}).
\end{proof}

\begin{remark}
	There exist matrices over $\mathbb{R}$ for which~\eqref{equ:limsup-kronecker} is false (see~\cite[last paragraph of p.~1532]{MR2889254}).
	However, Xiao and Xu~\cite{MR2889254} proved that if the limit is replaced by a limit superior then~\eqref{equ:limsup-kronecker} holds for all matrices over $\mathbb{C}$.
\end{remark}

\section{Proof of Theorem~\ref{thm:main}}

Let $p$ be a positive integer.
In order to prove Theorem~\ref{thm:main}, we proceed as follows.

\noindent
For binary words
\begin{equation*}
	x^{(1)} = x_1^{(1)}\cdots x_\ell^{(1)},\;
	\dots ,\;
	x^{(p)} = x_1^{(p)}\cdots x_\ell^{(p)},\; \text{ and }\;
	y = y_1\cdots y_\ell,
\end{equation*}
of the same length $\ell$, define $x^{(1)} \times \cdots \times x^{(p)} \times y$ to be the word of $(p+1)$-tuples
\begin{equation*}
	(x^{(1)}_1, \dots, x^{(p)}_1, y_1) \cdots (x^{(1)}_\ell, \dots, x^{(p)}_\ell, y_\ell) .
\end{equation*}
First, we construct a deterministic finite automaton $\mathcal{A}_p$ that takes as input $x^{(1)} \times \cdots \times x^{(p)} \times y$ and accepts if and only if
\begin{equation}\label{equ:accept-p}
	\big[x^{(1)}\big]_F = \cdots = \big[x^{(p)}\big]_F = [y]_F \quad \text{and} \quad y \in 0^* C_F .
\end{equation}
Hence, by the uniqueness of the canonical Fibonacci representation, it follows easily that the number of words of length $\ell$ that are accepted by $\mathcal{A}_p$ is equal to $S^{(p)}_F(f_{\ell+1})$.

Second, we prove that the graph of $\mathcal{A}_p$ is strongly connected and aperiodic.
Consequently, by Theorem~\ref{thm:perron-frobenius-automaton}, we get that
\begin{equation}\label{equ:SpFfell1}
	S^{(p)}_F(f_{\ell+1}) \sim c_p \lambda_p^\ell ,
\end{equation}
as $\ell \to +\infty$, where $c_p > 0$ is a real number and $\lambda_p$ is the Perron--Frobenius eigenvalue of the transition matrix of $\mathcal{A}_p$.

Finally, for every positive integer $N$, let $\ell$ be the unique positive integer such that
\begin{equation*}
	f_\ell \leq N < f_{\ell + 1} .
\end{equation*}
Hence, we have
\begin{equation}\label{equ:squeeze}
	S^{(p)}_F(f_{\ell}) \leq S^{(p)}_F(N) < S^{(p)}_F(f_{\ell+1}) .
\end{equation}
Moreover, the Binet formula yields $N \asymp \varphi^\ell$.
Consequently, from~\eqref{equ:SpFfell1} and~\eqref{equ:squeeze} we get that
\begin{equation*}
	S^{(p)}_F(N) \asymp_p \lambda_p^\ell = \big(\varphi^\ell\big)^{(\log \lambda_p) / \! \log \varphi} \asymp_p N^{(\log \lambda_p) / \! \log \varphi} ,
\end{equation*}
as desired.
Note that $\lambda_p > 1$; otherwise~\eqref{equ:main-asymp} could not be satisfied.

It remains to construct $\mathcal{A}_p$ and to prove that its graph is strongly connected and aperiodic.
The construction of $\mathcal{A}_p$ proceeds as follows.
We begin by constructing a deterministic finite automaton $\mathcal{B}_p$ that consists of $p$ copies of the Berstel's automaton $\mathcal{B}$ that run in parallel on inputs $(x^{(1)}, y)$, \dots, $(x^{(p)}, y)$.
More precisely, we define $\mathcal{B}_p$ as follows.
\vspace{0.5em}
\begin{itemize}
	\setlength\itemsep{0.5em}
	
	\item The states of $\mathcal{B}_p$ are the elements of $\{a,b,c,d\}^p$.
	We write the states of $\mathcal{B}_p$ as $p$-tuples or as words of $p$ letters.
	
	\item The initial state of $\mathcal{B}_p$ is $a^p$.
	
	\item The accepting states of $\mathcal{B}_p$ are $a^p$ and $d^p$.
	
	\item There is a transition labeled $(x_1, \dots, x_p, y)$ from the state $(s_1, \dots, s_p)$ to the state $(s_1^\prime, \dots, s_p^\prime)$ of $\mathcal{B}_p$ if and only if for each $i \in \{1, \dots, p\}$ there is a transition labeled $(x_i, y)$ from the state $s_i$ to the state $s_i^\prime$ of $\mathcal{B}$.
\end{itemize}
By construction, and by the properties of Berstel's automaton $\mathcal{B}$, it follows easily that $\mathcal{B}_p$ accepts the input $x^{(1)} \times \cdots \times x^{(p)} \times y$ if and only if~\eqref{equ:accept-p} holds.

Note that $\mathcal{B}_p$ has $4^p$ states.
We shall prove that (for large $p$) most of these states are not accessible.
More precisely, we shall prove that the set of accessible states of $\mathcal{B}_p$ is equal to
\begin{equation*}
	\mathcal{S}_p := \{a, b\}^p \cup \{a, c\}^p \cup \{b, d\}^p .
\end{equation*}
For all $s_1, s_2 \in \{a, b, c, d\}$, put $\{s_1, s_2\}^{\bullet p} := \{s_1, s_2\}^p \setminus \{s_1^p, s_2^p\}$.
The following claims on the states and the transitions of $\mathcal{B}_p$ can be easily checked.
\vspace{0.5em}
\begin{enumerate}[label={\footnotesize(t\arabic*)},labelindent=1em,itemsep=0.5em,topsep=0.5em]
	\setlength\itemsep{0.5em}
	
	\item\label{ite:t1} The transitions starting from $a^p$ are of the following two kinds.
	\vspace{0.5em}
	\begin{enumerate}[label={\footnotesize(t1\roman*)},labelindent=1em,itemsep=0.5em,topsep=0.5em]
		\setlength\itemsep{0.5em}
		
		\item A unique transition labeled $(0, \dots, 0)$ and going to $a^p$ itself.
		
		\item Transitions labeled $(x_1, \dots, x_p, 1)$ and sending $a^p$ to the state $s = (s_1, \dots, s_p) \in \{b, d\}^p$, where each $x_i \in \{0, 1\}$ is arbitrary, $s_i = b$ if $x_i = 0$, and $s_i = d$ if $x_i = 1$.
	\end{enumerate}
	
	\item\label{ite:t2} There is only one transition starting from $b^p$.
	This transition goes to $c^p$.
	
	\item\label{ite:t3} The transitions starting from $c^p$ are of the following three kinds.
	\vspace{0.5em}
	\begin{enumerate}[label={\footnotesize(t3\roman*)},labelindent=1em,itemsep=0.5em,topsep=0.5em]
		\setlength\itemsep{0.5em}
		
		\item Transitions labeled $(x_1, \dots, x_p, 0)$ and sending $c^p$ to the state $(s_1, \dots, s_p) \in \{a, b\}^{\bullet p}$, where
		$x_1, \dots, x_p \in \{0,1\}$ are arbitrary but not all equal, $s_i = b$ if $x_i = 0$, and $s_i = a$ if $x_i = 1$.
		
		\item A unique transition with labels $(0, \dots, 0)$ and $(1, \dots, 1)$ that goes to $b^p$.
		
		\item A unique transition labeled $(1, \dots, 1, 0)$ that goes to $a^p$.
		
	\end{enumerate}
	
	\item\label{ite:t4} There is only one transition starting from $d^p$.
	This transition goes to $a^p$.
	
	\item\label{ite:t5} For each state $s = (s_1, \dots, s_p) \in \{a, b\}^{\bullet p}$ there is a unique transition starting from $s$.
	This transition is labeled $(x_1, \dots, x_p, 0)$ and sends $s$ to the state $(s_1^\prime, \dots, s_p^\prime) \in \{a, c\}^{\bullet p}$, where $x_i = 0$ and $s_i^\prime = a$ if $s_i = a$, while $x_i = 1$ and  $s_i^\prime = c$ if $s_i = b$.
	
	\item\label{ite:t6} For each state $s = (s_1, \dots, s_p) \in \{a, c\}^{\bullet p}$ the transitions starting from $s$ are of the following two kinds.
	\vspace{0.5em}
	\begin{enumerate}[label={\footnotesize(t6\roman*)},labelindent=1em,itemsep=0.5em,topsep=0.5em]
		\setlength\itemsep{0.5em}
		\item Transitions labeled $(x_1, \dots, x_p, 0)$ and sending $s$ to the state $(s_1^\prime, \dots, s_p^\prime) \in \{a, b\}^p$, where if $s_i = a$ then $x_i = 0$ and $s_i^\prime = a$; while if $s_i = c$ then $x_i = 0$ and $s_i^\prime = b$, or $x_i = 1$ and $s_i^\prime = a$
		
		\item Transitions labeled $(x_1, \dots, x_p, 1)$ and sending $s$ to the state $(s_1^\prime, \dots, s_p^\prime) \in \{b, d\}^p$, where if $s_i = a$ then $x_i = 0$ and $s_i^\prime = b$, or $x_i = 1$ and $s_i^\prime = d$; while if $s_i = c$ then $x_i = 1$ and $s_i^\prime = b$.
		
	\end{enumerate}
	
	\item\label{ite:t7} For each state $s = (s_1, \dots, s_p) \in \{b, d\}^{\bullet p}$ there is a unique transition starting from $s$.
	This transition is labeled $(x_1, \dots, x_p, 0)$ and sends $s$ to the state $(s_1^\prime, \dots, s_p^\prime) \in \{a, c\}^{\bullet p}$, where $x_i = 0$ and $s_i^\prime = a$ if $s_i = d$, while $x_i = 1$ and $s_i^\prime = c$ if $s_i = b$.
	
\end{enumerate}

The transitions described in~\ref{ite:t1}--\ref{ite:t7} are depicted in Figure~\ref{fig:transitions-schematic}.

\begin{figure}[ht]
	\centering
	\vspace{0.5em}
	\includegraphics{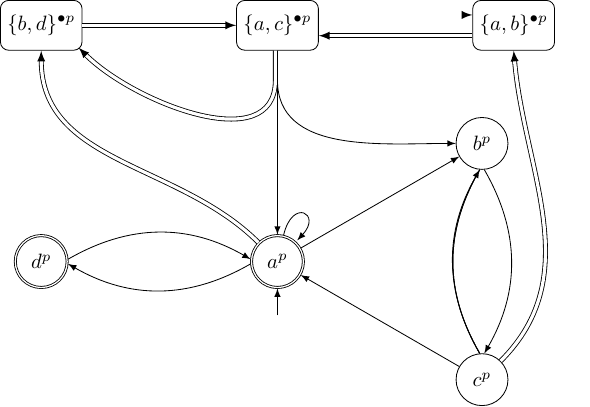}
	\caption{Transitions from states in $\mathcal{S}_p$.}
	\label{fig:transitions-schematic}
\end{figure}

From~\ref{ite:t1}--\ref{ite:t7}, we get that each transition starting from a state in $\mathcal{S}_p$ go to a state in $\mathcal{S}_p$.
Since $\mathcal{S}_p$ contains the initial state of $\mathcal{B}_p$, it follows that the set of accessible states of $\mathcal{B}_p$ is a subset of $\mathcal{S}_p$.
Furthermore, from~\ref{ite:t1}--\ref{ite:t7} it follows easily that the graph of $\mathcal{B}_p$ restricted to the states in $\mathcal{S}_p$ is strongly connected.
Hence, the set of accessible states of $\mathcal{B}_p$ is equal to $\mathcal{S}_p$, as previously claimed.

At this point, we define $\mathcal{A}_p$ as the deterministic finite automaton obtained by removing from $\mathcal{B}_p$ the states that are not accessible.
For example, the automaton $\mathcal{A}_2$ is depicted in Figure~\ref{fig:berstel2-automaton}.
We already observed that the graph of $\mathcal{A}_p$ is strongly connected.
Moreover, it follows easily that the graph of $\mathcal{A}_p$ is aperiodic, since the initial state has a loop.

The proof is complete.

\begin{remark}
	The automaton $\mathcal{A}_p$ has $3 \cdot 2^p - 2$ states but (for large $p$) many of these states are indistinguishable.
	By merging the indistinguishable states of $\mathcal{A}_p$, it is possible to obtain a minimized automaton with only $2^{p+1}$ states.
\end{remark}

\section{Proof of Theorem~\ref{thm:lambda-p}}

We let $\mathcal{A}_p$, $\mathcal{B}_p$, $\mathcal{S}_p$, and $\lambda_p$ be defined as in the proof of Theorem~\ref{thm:main}.
In particular, recall that $\mathcal{S}_p$ is the set of accessible states of $\mathcal{B}_p$, and that $\lambda_p$ is the spectral radius of the transition matrix of $\mathcal{A}_p$.
Let $\mathcal{S}_p^\prime$, respectively $\mathcal{S}_p^{\prime\prime}$, be the set of nonaccessible states $s = (s_1, \dots, s_p)$ of $\mathcal{B}_p$ such that $d$ does not belong, respectively belongs, to $\{s_1, \dots, s_p\}$.
Note that $\mathcal{S}_p, \mathcal{S}_p^\prime, \mathcal{S}_p^{\prime\prime}$ is a partition of $\{a, b, c, d\}^p$.

The following claims on the transitions of $\mathcal{B}_p$ can be easily checked.
\vspace{0.5em}
\begin{enumerate}[label={\footnotesize(t\arabic*)},labelindent=1em,itemsep=0.5em,topsep=0.5em]
	\setlength\itemsep{0.5em}
	\setcounter{enumi}{7}
	
	\item\label{ite:t8} There is no transition from a state in $\mathcal{S}_p$ to a state in $\mathcal{S}_p^\prime \cup \mathcal{S}_p^{\prime\prime}$.
	
	\item\label{ite:t9} For each state $s \in \mathcal{S}_p^\prime$ the transitions starting from $s$ are of the following two kinds.
	\vspace{0.5em}
	\begin{enumerate}[label={\footnotesize(t9\roman*)},labelindent=1em,itemsep=0.5em,topsep=0.5em]
		\setlength\itemsep{0.5em}
		
		\item Transitions sending $s$ to a state $s^\prime \notin \mathcal{S}_p^{\prime\prime}$ (not necessarily belonging to $\mathcal{S}_p^\prime$) such that $s^\prime$ contains more $a$'s than $s$.
		
		\item\label{ite:t9ii} Transitions sending $s = (s_1, \dots, s_p)$ to a state $s^\prime = (s_1^\prime, \dots, s_p^\prime) \in \mathcal{S}_p^\prime$ such that $s^\prime_i$ is equal to $a$, $b$, or $c$ if $s_i$ is equal to $a$, $c$, or $b$, respectively.
	\end{enumerate}
	
	\item\label{ite:t10} Every transition starting from $\mathcal{S}_p^{\prime\prime}$ does not have destination in $\mathcal{S}_p^{\prime\prime}$.
	
\end{enumerate}

\noindent
We sort the states of $\mathcal{B}_p$ according to the following rules.
\vspace{0.5em}
\begin{enumerate}[label={\footnotesize(o\arabic*)},labelindent=1em,itemsep=0.5em,topsep=0.5em]
	\setlength\itemsep{0.5em}
	
	\item\label{ite:ord1} The states in $\mathcal{S}_p$ come first (in whatever order).
	
	\item\label{ite:ord2} Then the states in $\mathcal{S}_p^\prime$ follow, sorted so that:
	\vspace{0.5em}
	\begin{enumerate}[label={\footnotesize(o2\roman*)},labelindent=1em,itemsep=0.5em,topsep=0.5em]
		\setlength\itemsep{0.5em}
		
		\item the number of $a$'s in each state is less than or equal to the number of $a$'s in the previous state;
		
		\item states connected by a transition in~\ref{ite:t9ii} are consecutive.
	\end{enumerate}
	
	\item\label{ite:ord3} The states in $\mathcal{S}_p^{\prime\prime}$ come last.
\end{enumerate}

Let $T_p$ and $U_p$ be the transition matrices of $\mathcal{A}_p$ and $\mathcal{B}_p$, respectively, according to the ordering of states defined by~\ref{ite:ord1}--\ref{ite:ord3}.
In light of~\ref{ite:t8}--\ref{ite:t10}, and since $\mathcal{A}_p$ is obtained from $\mathcal{B}_p$ by removing the states that are not in $\mathcal{S}_p$, we get that
\begin{equation*}
	U_p = \begin{pmatrix}\begin{array}{c|c|c}
			T_p & \mathbf{0} & \mathbf{0} \\[2pt] \hline
			\rule{0pt}{\normalbaselineskip} * & T_p^\prime & \mathbf{0} \\[2pt] \hline
			\rule{0pt}{\normalbaselineskip} * & * & \mathbf{0}
	\end{array}\end{pmatrix}, \quad\text{ with }\quad
	T_p^\prime := \begin{pmatrix}\begin{array}{c|c|c|c|c}
			\begin{matrix} 0 & 1 \\ 1 & 0 \end{matrix} & \mathbf{0} & \mathbf{0} & \cdots & \mathbf{0} \\ \hline
			\rule{0pt}{1.7em} * & \begin{matrix} 0 & 1 \\ 1 & 0 \end{matrix} & \mathbf{0} & \,\cdots & \mathbf{0} \\ \hline
			\rule{0pt}{1.7em} * & * & \begin{matrix} 0 & 1 \\ 1 & 0 \end{matrix} & \,\cdots & \mathbf{0} \\ \hline
			\rule{0pt}{2em} \vdots & \vdots & \vdots & \,\ddots & \vdots \\[5pt] \hline
			\rule{0pt}{1.7em} * & * & * & \,\cdots & \begin{matrix} 0 & 1 \\ 1 & 0 \end{matrix}
	\end{array}\end{pmatrix} ,
\end{equation*}
where the $\mathbf{0}$'s are zero matrices, and the $\mathbf{*}$'s denote matrices of the right sizes.

Therefore, the spectrum of $U_p$ is equal to the union of $\{-1, 0, +1\}$ and the spectrum of $T_p$.
Recalling that $\lambda_p > 1$, we get that $\rho(U_p) = \rho(T_p) = \lambda_p$.

\begin{figure}[ht]
	\centering
	\includegraphics{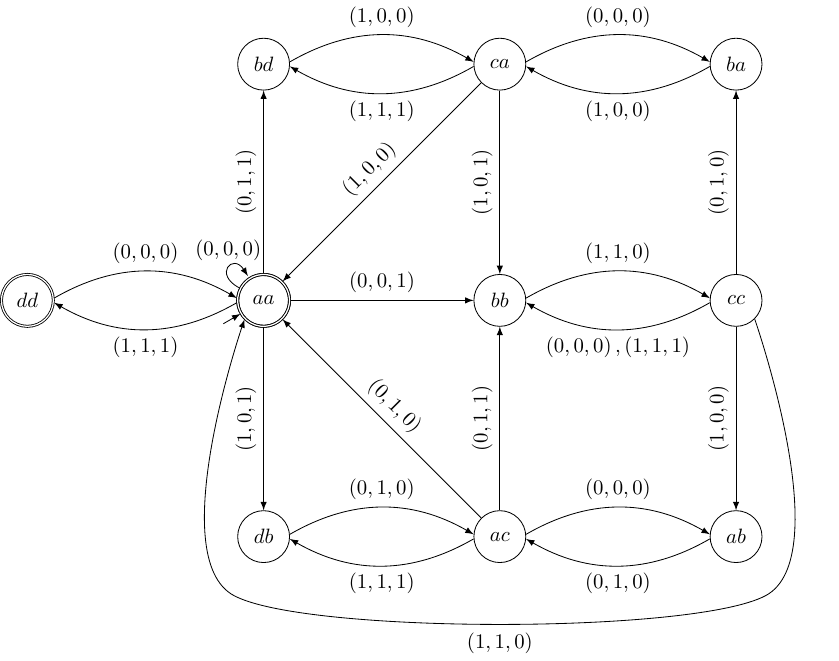}
	\caption{The automaton $\mathcal{A}_2$.}
	\label{fig:berstel2-automaton}
\end{figure}

For each $y \in \{0, 1\}$, let $V_y$ be the $4 \times 4$ matrix whose entry of the $i$th row and $j$th column is equal to the number of transitions labeled $(*,y)$ (where $*$ is any symbol) that go from the $i$th state to the $j$th state of the Berstel automaton $\mathcal{B}$.
Explicitly, we have that
\begin{equation*}
	V_0 = \begin{pmatrix}
		1 & 0 & 0 & 0 \\
		0 & 0 & 1 & 0 \\
		1 & 1 & 0 & 0 \\
		1 & 0 & 0 & 0 \\
	\end{pmatrix} \quad \text{and} \quad
	V_1 = \begin{pmatrix}
		0 & 1 & 0 & 1 \\
		0 & 0 & 0 & 0 \\
		0 & 1 & 0 & 0 \\
		0 & 0 & 0 & 0 \\
	\end{pmatrix} .
\end{equation*}

The next lemma provides the generalized spectral radius of $V_0$ and $V_1$.

\begin{lemma}\label{lem:rho-T0-T1}
	We have that $\rho(V_0, V_1) = \varphi^{1/2}$.
\end{lemma}
\begin{proof}
	For every binary word $x = x_1 \cdots x_k$, let $V_x := V_{x_1} \cdots V_{x_k}$.
	For all positive integers $h$, we have that
	\begin{equation*}
		\big(\rho(V_{(0001)^h})\big)^{1/(4h)} = \big(\rho(V_{0001}^h)\big)^{1/(4h)} = \big(\rho(V_{0001})^h\big)^{1/(4h)} = \big(\rho(V_{0001}))^{1/4} = \varphi^{1/2} .
	\end{equation*}
	Hence, we get that $\rho(V_0, V_1) \geq \varphi^{1/2}$.
	Thus it suffices to prove that $\rho(V_x)^{1/k} \leq \varphi^{1/2}$ for every binary word $x = x_1 \cdots x_k$.
	Note that we have
	\begin{equation*}
		\big(\rho(V_{0^k})\big)^{1/k} = \big(\rho(V_0^k)\big)^{1/k} = \big(\rho(V_0)^k\big)^{1/k} = \rho(V_0) = 1 .
	\end{equation*}
	Hence, we can assume that $x \neq 0^k$.
	It is well known that, if $A$ and $B$ are $n \times n$ complex matrices, then $AB$ and $BA$ have the same eigenvalues.
	Consequently, by applying a circular shift to $x$, we can assume that $x_k = 1$.
	Moreover, since $V_1^2 = \mathbf{0}$, we can assume that $x$ does not have two consecutive $1$'s.
	Again since $V_1^2 = \mathbf{0}$, we get that if $x_1 = 1$ then $V_x^2 = \mathbf{0}$, and consequently $\rho(V_x) = 0$.
	Hence, we can assume that $x_1 = 0$ (and so $k \geq 2$).
	Therefore, we have that $x = (0^{k_1 - 1}1) \cdots (0^{k_s - 1}1)$, for some integers $k_1, \dots, k_s \geq 2$ such that $k_1 + \cdots + k_s = k$.
	
	Suppose that $\lambda \neq 0$ is an eigenvalue of $V_x$, and let $v$ be the corresponding eigenvector, so that $v V_x = \lambda v$.
	Since $x_k = 1$, we get that $v = (0 \; v_2 \; 0 \; v_4)$ for some $v_2, v_4 \in \mathbb{C}$.
	It follows easily by induction that
	\begin{equation*}
		(0 \; w_2 \; 0 \; w_3) \, V_{0^{h-1} 1}
		= (0 \; w_2 \; 0 \; w_3)
		\begin{pmatrix}
			0 & 0 & 0 & 0 \\
			0 & \lfloor h / 2 \rfloor & 0 & \lfloor (h-1) / 2 \rfloor \\
			0 & 0 & 0 & 0 \\
			0 & 1 & 0 & 1
		\end{pmatrix} ,
	\end{equation*}
	for all $w_2, w_4 \in \mathbb{C}$ and for every integer $h \geq 2$.
	Therefore, we ge that $v V_x = \lambda v$ is equivalent to $(v_2 \; v_4) Z_{k_1} \cdots Z_{k_s} = \lambda (v_2 \; v_4)$, where
	\begin{equation*}
		Z_h :=
		\begin{pmatrix}
			\lfloor h / 2 \rfloor & \lfloor (h-1) / 2 \rfloor \\
			1 & 1
		\end{pmatrix}
	\end{equation*}
	for every integer $h \geq 2$.
	Consequently, we obtain that
	\begin{equation}\label{equ:bound-rho-Tx}
		\rho(V_x) \leq \rho(Z_{k_1} \cdots Z_{k_s}) \leq \|Z_{k_1} \cdots Z_{k_s}\| \leq \|Z_{k_1}\| \cdots \|Z_{k_s}\| ,
	\end{equation}
	where $\|\cdot\|$ is the \emph{spectral norm}, which is defined by $\|A\| := \sqrt{\!\rho(A^{\mathsf{H}} A)}$ for every square matrix $A$ over $\mathbb{C}$, with $A^\mathsf{H}$ denoting the conjugate transpose.
	
	We claim that $\|Z_h\|^{1/h} \leq \varphi^{1/2}$ for every integer $h \geq 2$.
	Indeed, since the eigenvalues of $Z_h^{\mathsf{H}}\, Z_h$ are nonnegative real numbers, we have that
	\begin{equation*}
		\|Z_h\|^{1/h} = \big(\rho(Z_h^{\mathsf{H}}\, Z_h)\big)^{1/(2h)} \leq \big(\!\trace(Z_h^{\mathsf{H}}\, Z_h)\big)^{1/(2h)} \leq \left(\frac{h^2+4}{2}\right)^{1/(2h)} \leq \varphi^{1/2} ,
	\end{equation*}
	for every integer $h \geq 7$.
	Then the claim can be checked for $h \in \{2, \dots, 7\}$.
	
	Therefore, from~\eqref{equ:bound-rho-Tx} we get that
	\begin{align*}
		\big(\rho(V_x)\big)^{1/k} &\leq \big(\|Z_{k_1}\| \cdots \|Z_{k_s}\|\big)^{1/k} = \big(\|Z_{k_1}\|^{1/k_1}\big)^{k_1/k} \cdots \big(\|Z_{k_s}\|^{1/k_s}\big)^{k_s/k} \\
		&\leq \big(\varphi^{1/2}\big)^{k_1/k} \cdots \big(\varphi^{1/2}\big)^{k_s/k} = \big(\varphi^{1/2}\big)^{(k_1 + \cdots + k_2)/k} = \varphi^{1/2} ,
	\end{align*}
	as desired.
\end{proof}

Let $U_p^{\textsf{lex}}$ be the transition matrix of $\mathcal{B}_p$, where the states of $\mathcal{B}_p$ are sorted in lexicographic order.
Note that the matrices $U_p^{\textsf{lex}}$ and $U_p$ are similar, and  consequently $\rho(U_p^{\textsf{lex}}) = \rho(U_p) = \lambda_p$.
By the construction of $\mathcal{B}_p$, we have that $U_p^{\textsf{lex}} = V_0^{\otimes p} + V_1^{\otimes p}$.
Hence, by Theorem~\ref{thm:kronecker-power} and Lemma~\ref{lem:rho-T0-T1}, we get that
\begin{equation*}
	\lim_{p \to +\infty} \lambda_p^{1/p} = \lim_{k \to +\infty} \big(\rho(U_p^{\textsf{lex}})\big)^{1/p} = \lim_{k \to +\infty} \Big(\rho\big(V_0^{\otimes p} + V_1^{\otimes p}\big)\Big)^{1/p} = \rho(V_0, V_1) = \varphi^{1/2} .
\end{equation*}
The proof is complete.

\section*{Statements and Declarations}

\subsection*{Competing Interests}
The author declare that he has no known competing financial interests or personal relationships that could have appeared to influence the work reported.

\subsection*{Data availability statement}
No new data were created or analysed in this study.
Data sharing is not applicable to this article.

\section*{Acknowledgments} 
The author would like to thank the anonymous referees for carefully reading the paper.

\bibliographystyle{amsplain}


\begin{dajauthors}
\begin{authorinfo}[casa]
  Carlo Sanna\\
  Department of Mathematical Sciences, Politecnico di Torino\\
  Corso Duca degli Abruzzi 24, 10129 Torino, Italy\\
  carlo\imagedot{}sanna\imageat{}polito\imagedot{}it
\end{authorinfo}
\end{dajauthors}

\end{document}